\documentclass{amsart}

\usepackage{graphicx}
\usepackage{psfrag}
\usepackage{amsthm}
\usepackage{amscd}
\usepackage{amsmath}
\usepackage{amsfonts}
\usepackage{amssymb}
\usepackage{amstext}
\usepackage{color}
\def\BB{\mathcal{B}}
\def\JJ{\mathcal{J}}

\def\UU{\mathcal{U}}
\def\integer{{\mathbb{Z}}}
\def\real{{\mathbb{R}}}
\def\nat{{\mathbb{N}}}
\def\Ker{{\mathrm{Ker}\, }}

\def\intt{{\mathrm{int}}\, }

\begin{document}
\author{Viviane Baladi and Daniel Smania}

\address{D.M.A., UMR 8553, \'Ecole Normale Sup\'erieure,  75005 Paris, France}
\email{viviane.baladi@ens.fr}
\urladdr{www.dma.ens.fr/$\sim$baladi/}

\address{
Departamento de Matem\'atica,
ICMC-USP, Caixa Postal 668,  S\~ao Carlos-SP,
CEP 13560-970
S\~ao Carlos-SP, Brazil}
\email{smania@icmc.usp.br}
\urladdr{www.icmc.usp.br/$\sim$smania/}
\date{\today} 
\title{Smooth deformations of piecewise expanding unimodal maps}
\begin{abstract} In the space of $C^k$ piecewise expanding unimodal maps, $k\geq 1$, we characterize  the $C^1$ smooth families of maps  where the topological dynamics does not change (the ``smooth deformations")  as the families tangent to a continuous distribution of codimension-one subspaces (the ``horizontal" directions)  in that space.  Furthermore such codimension-one subspaces are  defined as the kernels of  an
explicit class of linear functionals. As a consequence we show the existence of $C^{k-1+Lip}$ deformations tangent to every  given $C^k$ horizontal  direction, for $k\ge 2$.  
\end{abstract}
\thanks{V.B. is partially supported by ANR-05-JCJC-0107-01.
D.S. is partially supported by CNPq 470957/2006-9
and 310964/2006-7, FAPESP 2003/03107-9. V.B. did part of this work
at UMI 2924 CNRS-IMPA, Rio de Janeiro.}
\maketitle

\newcommand{\co}{\mathbb{C}}
\newcommand{\incl}[1]{i_{U_{#1}-Q_{#1},V_{#1}-P_{#1}}}
\newcommand{\inclu}[1]{i_{V_{#1}-P_{#1},\co-P}}
\newcommand{\func}[3]{#1\colon #2 \rightarrow #3}
\newcommand{\norm}[1]{\left\lVert#1\right\rVert}
\newcommand{\norma}[2]{\left\lVert#1\right\rVert_{#2}}
\newcommand{\hiper}[3]{\left \lVert#1\right\rVert_{#2,#3}}
\newcommand{\hip}[2]{\left \lVert#1\right\rVert_{U_{#2} - Q_{#2},V_{#2} - P_{#2}}}
\newtheorem{lem}{Lemma}[section]
\newtheorem{defi}{Definition}[section]
\newtheorem{rem}{Remark}[section]
\newtheorem{mle}{Main Lemma}[section]
\newtheorem{mth}{Main Theorem}
\newtheorem{thm}{Theorem}
\newtheorem{theorem}{Theorem}[section]
\newtheorem{lemma}[theorem]{Lemma}
\newtheorem{corollary}[theorem]{Corollary}
\newtheorem{remark}{\it Remark\/}
\newtheorem{example}{\it Example\/}

\newtheorem{proposition}{Proposition}[section]

\newcommand{\comment}[1]{}

\newtheorem{cor}{Corollary}[section]
\newtheorem{pro}{Proposition}[section]
\newtheorem{conj}{Conjecture}
\setcounter{tocdepth}{1}

\section{Introduction}

The topological class of a dynamical system $f$  is the set of all maps topologically conjugate with $f$. A {\it smooth deformation}  of a dynamical system $f_0$  is a smooth family  of dynamical systems $t\mapsto f_t$ inside the topological class of $f_0$.  We
also  say that  a smooth deformation $f_t$ is a family with ``no bifurcations." Deciding whether or not there are bifurcations in a  family is one of the primary problems concerning dynamical systems.

In the theory of complex dynamical systems, specially for rational functions, this type of study was very sucessful. One of the most  powerful tools in complex dynamics are the quasiconformal methods: quasiconformal maps, quasiconformal vector fields, and holomorphic motions.  For example, they allow us to easily find a holomorphic  deformation between two holomorphic dynamical systems which are conjugate by a quasiconformal map, using the so-called Beltrami path.  Beltrami paths are examples of {\it holomorphic motions,} whose importance in complex dynamics can not be overstated since the time they were introduced in the seminal work by Ma\~n\'e, Sad and Sullivan \cite{MSS}.  Holomorphic motions are a key tool in the characterization of structurally stable rational maps and families of rational maps  with no bifurcations \cite{MSS} (see also \cite{MS}). The  study of the regularity of  hybrid classes of  quadratic-like maps \cite{L}  and topological classes of  analytic unimodal maps \cite{ALM} also depends heavily on quasiconformal methods.

Unfortunately quasiconformal methods do not seem to be  applicable for real one-dimensional maps which do not have a holomorphic extension to the complex plane, like  piecewise expanding  $C^2$ unimodal maps (see Section ~2 for formal definitions). 

In our first main result, Theorem \ref{equiv} (Section ~4), we characterize all smooth families in the space of piecewise expanding unimodal maps which  are smooth deformations: they are precisely the families  tangent  to a continuous distribution of codimension-one subspaces in that space. Following the notation in \cite{L}, these subspaces will be called ``horizontal directions."
See Section 3 for the definition of the linear functional $J(f,\cdot)$ whose kernel defines
the horizontal directions.  

We observe that for families of smooth unimodal maps, the condition  $$J(f_{t_0},\partial_t f_t|_{t=t_0})\neq 0$$  is the  ``nondegeneracy condition" for a family $f_t$ at a Collet-Eckmann parameter $f_{t_0}$ that  appeared in a generalization of Jakobson's Theorem by Tsujii \cite{T}. On the other hand, the condition $J(f, v)=0$ for $v$ to be horizontal is the same
which is well-known for  smooth unimodal maps satisfying certain summability condition (see e.g. \cite{av}, \cite{ALM}). This condition
first appeared in the context of piecewise expanding maps in \cite{Ba}.

One can wonder if there {\it exist}   deformations of a given piecewise expanding unimodal map which are non trivial, i.e., so that the $f_t$ are not smoothly conjugate to
$f_0$ . (Not only because this is an intrisically natural question, but also because it
recently became clear that this is crucial to understand some dynamically
defined quantities, see below.)
We  answer this question in Theorem \ref{create} (Section~5).  In particular,  for each ``good" piecewise expanding unimodal map $f_0$ and each horizontal direction $v$, we  construct a smooth deformation of $f_0$ tangent to $v$ at $f_0$, i.e.,
$\partial_t f_t|_{t=0}=v$. 
In other words,  Theorem~ \ref{create} shows that the theory of smooth
deformations is very rich, since there are plenty of deformations of a piecewise expanding unimodal map  in  ``horizontal" directions.

In both theorems we heavily use  ``smooth motions," that is, we exploit
the fact that the conjugacies $h_t$ depend smoothy on $t$.
In Theorem \ref{equiv}, $B\Rightarrow C$  (see \cite{BS}), we use them in the phase space,  and in Theorem~ \ref{create}, in the parameter space.

Given a smooth family of dynamical systems, one can ask how  dynamically defined  quantities,
such as the  average of a given observable with respect to  the SRB measures, the Lyapunov exponents, and the  Hausdorff dimension of invariant sets,
change along this family. Studying smoothness of these quantities can be a tricky issue.  For example, SRB measures are often described as  eigenvectors of Ruelle-Perron-Frobenius operators acting on infinite-dimensional spaces with a complicated structure.

In the case of piecewise expanding  maps,   H\"older continuity of SRB measures (for all exponents $<1$) has been known for a long time
\cite{Ke}. However any hope of higher regularity for families transversal to the topological class was annihilated  by the examples in \cite{Ba} (see also \cite{MM}). In order to  have a satisfactory  theory about smooth variation  of  dynamically defined quantities, at least in the case of  the SRB measure of piecewise expanding unimodal maps, it was recently
discovered that 
we need to restrict ourselves to families tangent to topological classes \cite{BS}.
Theorem ~ \ref{create} and its corollaries imply the
result announced as  Theorem 2.8 in \cite{BS}. This result was not used to obtain
the other claims in ~\cite{BS},
but it shows that there are plenty of families satisfying the restriction of tangency to the
topological class needed there.

\section{Preliminaries}

Denote $I=[-1,1]$ and $\mathbb{N}=\integer_+$.

For $k \ge 0$, we define 
the set $\mathcal{B}^{k}(I)$ of {\it piecewise $C^ k$ functions} to be
the linear space   of  continuous functions 
$f\colon I \to \mathbb{R}$ such that $f$ is $C^{k}$ on the intervals 
$[-1,0] $ and  $[0,1]$, with $f(1)=f(-1)$.
Then $\mathcal{B}^{k}(I)$ is a Banach space for the norm
$$|f|_k = \max \{|f|_{C^{k}[-1,0]}, |f|_{C^{k}[0,1]}  \},
\mbox{ where }
|f|_{C^{k}(Q)} = \max_{0\le i\leq k}  \{ |D^i f|_{L^\infty(Q)}\}.$$

For $k\ge 1$, we define the set  $\mathcal{U}^k$  of
{\it piecewise expanding $C^ k$ unimodal} maps 
to be the set  of maps  $f \in \mathcal{B}^k(I)$ such that 
\begin{itemize} 
\item[I.]{\it (Invariance of $\partial I$)} $f(-1)=f(1)=-1$.
\item[II.]{\it (Expanding condition)} $\inf_{x \in [-1,0]}  Df(x) > 1$ and 
$\sup_{x \in [0,1]}  Df(x) < -1$.
\item[III.]{\it (Invariance of $I$)}  $f(0)\leq 1$
(by I.--II. this implies $f(I)\subset I$).
\end{itemize} 
The set $\mathcal{U}^k$ is a convex subset of  the affine subspace $$\{ f \in \mathcal{B}^k(I) \text{ s.t. } f(-1)=f(1)=-1\} $$ and
$\mathcal{U}^k\cap \{ f \in  \mathcal{B}^k(I)\colon \ f(0) < 1\}$
is a convex and open set of the same affine subspace.
We call elements of $\mathcal{U}^1$ simply {\it piecewise expanding  unimodal maps.}
The point $c=0$ is called the {\it critical point} of a piecewise expanding
unimodal map. Set  
$$\lambda_f = \min_x |Df(x)| > 1. $$
The itinerary of  $x\in I$
for  a piecewise expanding unimodal map $f$ is the 
sequence 
$$(\sigma_0(x),\sigma_1(x),\sigma_2(x),\dots)  \in \{L,C,R\}^{\mathbb{N}}$$
such that  $\sigma_i(x)=L$ if $f^i(x) < c$, $\sigma_i(x)=C$ if $f^i(x)=c$, 
and  $\sigma_i(x)=R$ if $f^i(x) >  c$.
We write $\sigma_i=\sigma_i(c)$.

Let $1\le j \le k$, with $k$ an integer,
and $j$ either an integer or $j=k-1+Lip$.
A  {\it $C^j$ family} $f_t$ of  piecewise expanding $C^k$ unimodal
maps is a $C^j$ map
\begin{equation}
\label{families}
t\mapsto f_t \mbox{ from } [-\delta,\delta] \mbox{ to } \UU^k  .
\end{equation}
(In particular, for such a family, the map $(t,x)\mapsto f_t(x)$ 
is continuous on $[-\delta,\delta]\times I$, and it is $C^{j}$ on the sets
$[-\delta,\delta]\times [-1,0]$ and $ [-\delta,\delta]\times [0,1]$.)

If $f_t$ is a family of piecewise expanding unimodal  maps, we consider
$$R_t:=\{ (i,j) \in \nat \times \nat\colon \  f^i_t(c)=f^j_t(c) \ \mbox{and} \  i < j   \},$$
the set of {\it critical relations}  of $f_t$. Note that if 
the forward orbit of $c$ (also called {\it postcritical orbit})
is infinite then $R_t$ is empty.

We say that a piecewise expanding unimodal map $f$ is {\it good}  if either $c$
is not periodic or, writing $p\ge 2$ for the prime period of $c$, if
\begin{equation}\label{good}
|Df^{p-1}(f(c)) |\min \{ |Df^+(c)|,|Df^-(c) |\} > 2.
\end{equation}
A map $f\colon I \to I$ is {\it $\epsilon$-expansive} 
if for every interval $L \subset I$ there is $i\ge 1$ so that 
$$|f^i(L)|> \epsilon.$$
A piecewise expanding unimodal  map 
$f_0$ is {\it stably $\epsilon$-expansive} if every piecewise expanding 
unimodal  map 
$f$ 
close enough to $f_0$ (for $|\cdot|_1$) is $\epsilon$-expansive.
We give the easy proof of the following useful result for completeness:

\begin{pro} \label{deepness}
Let $f$ be a piecewise expanding  unimodal  map. Then  there exists
$\epsilon >0$ so that $f$ is  $\epsilon$-expansive.
If we assume furthermore that $f$ is good, then  there exists
$\epsilon >0$ so that $f$ is  stably $\epsilon$-expansive.
\end{pro}

\begin{proof}  
Choose $N_0$ such that
$\frac{1}{2} \lambda_f^{N_0-1} >  \lambda_f$,
and $\epsilon$ such that 
$$c \not \in \intt(f^i[-\epsilon,\epsilon]) \mbox{ for } i=1, \dots, N_0 .
$$
Then for every interval $Q \subset [-\epsilon,\epsilon]$ we have 
\begin{equation}\label{grown} |f^{N_0}(Q)| >  \lambda_f |Q|.\end{equation}
If the turning point is not periodic it is easy  to see that  (\ref{grown}) remains 
true for any small enough perturbation of $f$.    

Consider an interval $Q\subset I$ and suppose that $|f^i(Q)|<\epsilon$
for every $i\in \nat$.
Define  $n_0=0$ and $n_1,n_2,n_3,\dots$ in the following way: If 
$$c\not \in Q_s:=f^{n_0+n_1+\dots+ n_s}(Q)$$
define $n_{s+1}=1$. In this case
$$|Q_{s+1}|=|f(Q_s)|\geq \lambda_f |Q_s|.$$
Otherwise set $n_{s+1}=N_0$. Note that $Q_s \subset  [-\epsilon,\epsilon]$
and that  (\ref{grown}) implies
$$|Q_{s+1}|=|f^{N_0}(Q_s)|\geq \lambda_f |Q_s|,$$
so
$|Q_s|\geq \lambda_f^s|Q|$,
which implies that $|Q|=0$, proving that $f$ is $\epsilon$-expansive.

If $c$ is periodic but (\ref{good}) holds, the argument above can be easily modified 
to show stable
$\epsilon$-expansiveness.
\end{proof}


\section{The linear functional $J(f,\cdot)$}
\subsection{Definition and relation with the twisted cohomological equation}

We shall associate a bounded linear functional  $J(f,\cdot) \in (L^ \infty(I))^*$ 
to each  piecewise expanding unimodal  map $f$.
This functional will play a main role in this work.  
Let 
$v\colon I \to \mathbb{R}$
be a bounded function. If the critical point $c$ is not periodic, we define

 $$J(f,v)= \sum_{i=0}^\infty \frac{v(f^{i}(c))}{Df^{i}(f(c))}.$$
The above expression  is not well defined if  the critical point $c$ is periodic, since the 
derivative at the critical point does not exist. If $c$ has  prime period $p$ 
we  set

 $$J(f,v)= \sum_{i=0}^{p-1} \frac{v(f^{i}(c))}{Df^{i}(f(c))}.$$

Note that in both cases (non periodic and periodic critical points) we have
$$|J(f,v)|\leq\frac{ |v|_{L^\infty}}{1-\lambda_f^ {-1}}.$$
It is easy to see that $v\mapsto J(f,v)$ is not the zero functional on $C(I)$,
so for every $k\in \nat$,  by the density of 
$C^k(I)$ in $C(I)$,  there exists $v \in C^k(I)$ with $J(f,v)\not= 0$.

The meaning of the expression for $J(f,v)$ can be clarified by the following comments. 
Let $f_t$ be a $C^1$ family of  piecewise expanding $C^1$ unimodal maps such that $\partial_t f_t|_{t=0}=v$, $f_0=f$. 
(We shall sometimes call the argument $v$ of $J(f,v)$ a {\it vector field}.)
Then, for any $k\ge 1$,
if $f^j(x)\ne c$  for $1 \le j \le k-1$
 $$\partial_t f^{k}_t(x)|_{t=0} = \sum_{i=0}^{k-1} Df^{k-1-i}(f^{i+1}(x))\cdot v(f^i(x)) , $$
so if $f^j(x)\ne c$  for $1 \le j \le k-1$, then
 \begin{equation}\label{clarif}
  \frac{\partial_t f^{k}_t(x)|_{t=0} }{Df^{k-1}(f(x))}= \sum_{i=0}^{k-1} 
\frac{v(f^{i}(x)))}{Df^{i}(f(x))}  .
\end{equation}
So, if $c$ has prime period $k$, then
\begin{equation}\label{clarif1}
J(f,v)=  \frac{\partial_t f^{k}_t(c)|_{t=0} }{Df^{k-1}(f(c))} ,\quad 
(v=\partial_t f_t|_{t=0}, f=f_0),
\end{equation} 
and if $c$ is  not periodic for $f$, then
 \begin{equation}\label{clarif2}J(f,v)= \lim_{k \to \infty} 
\frac{\partial_t f^{k}_t(c)|_{t=0} }{Df^{k-1}(f(c))}, \quad
(v=\partial_t f_t|_{t=0}, f=f_0).\end{equation}
In other words, 
the  derivatives in the phase and parameter spaces along the critical orbit are related by $J(f,v)$.

We also mention  that if $v=X \circ f$ then
$J(f,v)=s_1^ {-1} \JJ(f,X)$,
where $s_1 < 0$ is the jump $s_1=-\lim_{x\to 1} \rho(x)$ at $1$
of the invariant density $\rho$ of $f$, and where $\JJ(f,X)$ was introduced in
\cite{Ba} and used in \cite{BS}. It was observed in \cite{BS}
(see also Proposition \ref{stce} below) that
elements $v$ of the kernel of $\JJ(f,\cdot)$ satisfy
$\sum_{i=0}^ \infty \frac{v(f^ i(c))}{ Df^ i (f(c))}=0$
if $c$ is not periodic and $\sum_{i=0}^{p-1} \frac{v(f^ i(c))}{ Df^ i (f(c))}=0$
if $c$ has prime period $p$.
Such $v\in \Ker( (\JJ(f,\cdot))$ 
deserve to be called
{\it horizontal vector fields,} by analogy with the theory
for smooth unimodal maps (\cite{L}, \cite{ALM}) and in view of the results
in \cite{BS} (in particular Corollary 2.6 and Remark 2.7 there).
Our Theorems \ref{equiv} and \ref{create} also justify this terminology.

We next recall the relation between $J(f,v)$ and the {\it twisted cohomological equation
(\ref{t})}
from \cite[Lemma 2.2]{BS}.  

\begin{pro} \label{stce} For every piecewise
expanding unimodal map $f$ and  $v\in L^\infty(I)$  the following holds: 
Let $\mathcal{D}$ be  the set of  $x \in I$ with a forward orbit that does not contain $c$. 
There exists a unique  bounded  function 
$\alpha\colon \mathcal{D}\to \mathbb{R}$
such that 
\begin{equation}\label{t} v(x) =\alpha \circ f(x) -Df(x)\cdot \alpha(x),\end{equation}
for every $x \in \mathcal{D}$. There exists a unique bounded function 
$\alpha\colon I \to \mathbb{R}$  such that $\alpha(c)=0$ and (\ref{t}) holds 
for every $x\neq c$.

Furthermore 
$J(f,v)=0 \mbox{ if and only if } v(c)=\alpha(f(c))$.
\end{pro}
\begin{proof}  
We refer to \cite[Lemma 2.2]{BS}. We just recall that
for each $x \in \mathcal{D}$ 
$$\alpha(x)= -\sum_{i=0}^{\infty} \frac{v\circ f^i(x)}{Df^{i+1}(x)}, $$
and for $x \in I \setminus \mathcal{D} $, with $x\not = c$,
setting   $k=\min \{i > 0\  \colon \  f^i(x)=c \}$,
$$ \alpha(x)=- \sum_{i=0}^{k-1}  \frac{v\circ f^i(x)}{Df^{i+1}(x)}.
$$
\end{proof}


\subsection{Continuity of $\Ker (J(f,\cdot))$}

Observe that  $f \mapsto J(f,v)$
is continuous at piecewise expanding unimodal  maps with  non periodic critical point:
 
\begin{pro}\label{kerno} Let $f_0 \in \mathcal{U}^1$ be a piecewise expanding unimodal  map.  
If the critical point of $f_0$ is not periodic
then 
\begin{itemize}
\item[A.] For every $\eta > 0$ there exists a neighborhood  $W$ of $f_0$ in $\mathcal{U}^1$ 
such that  
$|J(f,v)-J(f_0,v)|\leq \eta |v|_1$ for every $v \in \mathcal{B}^1(I)$ and $f\in W$.
\item[B.] For every  $v_0 \in \mathcal{B}^0(I)$ the  function
$f\mapsto J(f,v_0)$
is continuous at $f=f_0$, considering the $\mathcal{U}^1$ norm. 
\end{itemize}

\end{pro}
\begin{proof}[Proof of Claim A]
Taking $W$ small enough, 
we have
$\theta=\inf_{f \in W}  \inf_{x} |Df(x)| >1$. 
Let $N $ be such that 
$\frac{\theta^N}{1-\theta} < \frac{\eta}{8}$.
Reducing  $W$, if necessary, we can assume that $f^i(f(c))\neq c$ for every $f \in W$ and  $i \leq N$, 
and that
\begin{equation}\label{gg}
\sum_{i\leq N} \Big| \frac{1}{Df^i(f(c))} -\frac{1}{Df_0^i(f_0(c))}\Big|< \frac{\eta}{4}
\mbox{ and } \sum_{i\leq N} |f^i(c) -f_0^i(c)|\leq \frac{\eta}{4}.
\end{equation}
Then
$$
|J(f,v)-J(f_0,v)|\le
 \sum_{i< N} \Big| \frac{v(f^{i+1}(c))}{Df^{i+1}(f(c))}-
\frac{v(f^{i+1}_0(c))}{Df^{i+1}_0(f_0(c))}\Big| + \frac{\eta}{4} |v|_0.
$$
Estimating $\sum_{i< N} \Big| \frac{v(f^{i+1}(c))}{Df^{i+1}(f(c))}-
\frac{v(f^{i+1}_0(c))}{Df^{i+1}_0(f_0(c))}\Big|$ by
$$
 \sum_{i< N}  |v(f^{i+1}(c))| \ 
\Big|\frac{1}{Df^{i+1}(f(c))}-\frac{1}{Df^{i+1}_0(f_0(c))}\Big| +  \sum_{i< N} \frac{|v(f^{i+1}(c))-v(f^{i+1}_0(c))|}{|Df_0^{i+1}(f_0(c))|},
$$
we get the claim from (\ref{gg}) and our choice of $N$.
\end{proof}

 \begin{proof}[Proof of Claim B]
  We can assume that $|v_0|\leq 1$. Fix $\eta$, and let  $W$ and
$N$ be like in the proof of  Claim A. Reducing $W$ if necessary, we have
$$\sum_{i\leq N} |v_0(f^i(c)) -v_0(f_0^i(c))|\leq \frac{\eta}{4}
\,, \forall f \in W .$$
The calculations in the proof of Claim A imply that $|J(f,v)-J(f_0,v)|< \eta$.
\end{proof}

Continuity  of $f \mapsto J(f,v)$ fails at  maps $f_0$ with periodic critical points. 
However, to prove Theorem~\ref{create}, the next result 
(which, losely speaking, implies that
when $f_t \to f_0$ then $J(f_t,v)\to 0$ if and only if $J(f_0,v)=0$) will suffice:

\begin{pro} \label{kerperiodic}Let $f_0$ be a good piecewise expanding $C^1$ unimodal  map with 
periodic critical point of prime period $p_0$.  There exist $C_+,C_- > 0$, 
such that:
\begin{itemize}
\item[A.]   For every $\eta>0$, there exists a neighborhood $W$ of $f_0$ in $\mathcal{U}^1$ such that, setting
\begin{equation}\label{MM}
\mathcal{M}=\{ f \in W\colon \ f^{p_0}(c)=c, \ f^i(c)\not=c \ for \ i< p_0  \},
\end{equation}
the set $W\setminus \mathcal{M}$ has two connected components, $W_+$ and $W_-$, 
so that, for any $v\in \mathcal{B}^1(I)$, if $f \in \mathcal{M}$ then  $|J(f,v)-J(f_0,v)|\leq \eta |v|_1$,
if $f \in W_+$ then $| J(f,v)-C_+J(f_0,v)|\leq \eta |v|_1$,
  if $f \in W_-$ then $| J(f,v)-C_-J(f_0,v)|\leq \eta |v|_1$.

\item[B.] For every $v \in \mathcal{B}^0(I)$ and $\eta > 0$, there exists a neighborhood $W$ of $f_0$ in $\mathcal{U}^1$ such that $W\setminus \mathcal{M}$
(with $\mathcal{M}$ defined by (\ref{MM})) has two connected components, $W_+$ and $W_-$, so that if $f \in \mathcal{M}$ then  $|J(f,v)-J(f_0,v)|\leq \eta $,
if $f \in W_+$ then $| J(f,v)-C_+J(f_0,v)|\leq \eta $,
 if $f \in W_-$ then $| J(f,v)-C_-J(f_0,v)|\leq \eta $.
\end{itemize}
\end{pro}

A  consequence   of 
Propositions ~\ref{kerno} and ~\ref{kerperiodic} (B.) is  that 
$ \Ker ( J(f,\cdot))$ is a continuous distribution of codimension-one subspaces
for {\it any} good $f$:

\begin{cor}[Continuity of $\Ker (J(f,\cdot))$] \label{continuity} Let $f$ be a
good piecewise  expanding $C^1$ unimodal  map. Suppose that  $f_n$ is a sequence of piecewise 
expanding $C^1$ unimodal  maps with
$|f_n-f|_1\to 0$,
and that  $v_n\in \mathcal{B}^0(I)$ and $v\in\mathcal{B}^0(I)$ are  such that 
$$|v_n - v|_0\to 0 \mbox{ and }
J(f_n,v_n)=0 \, , \forall n
$$ 
then   $J(f,v)=0$.
\end{cor}

 \begin{proof} [Proof of Proposition~\ref{kerperiodic}]
Since $$J(f_0,\cdot): \{u \in \BB^1(I), u(-1)=u(1)=0\} \to \real$$ is a non trivial linear functional, there is  $w\in \BB^ 1(I)$,  with $w(-1)=w(1)=0$, so that 
$J(f_0,w) >0$. 
Define the subspace
$$K=\{ u \in \mathcal{B}^1(I), \ u(-1)=u(1)=0, \ J(f_0,u)=0\}.$$
We can identify a neighborhood of $f_0$ in $\mathcal{U}^ 1$ 
with  a neighborhood $\widetilde W$ of $(0,0)$ in  $K \times \mathbb{R}$
via $(u,a)\to f_0+u + aw$.
Consider the functional
 $F\colon \widetilde W \to \mathbb{R}$
 defined by 
 $$F(u,a)=(f_0+u+aw)^{p_0}(c)-c,$$
where $p_0$ denotes the prime period of $c$ for $f_0$. 
 
Note that $F$ is a $C^1$-function  and (recall (\ref{clarif}))
\begin{equation}\label{Imp1}
\partial_a F|_{(u,a)=(0,0)} =  Df_0^{p_0-1}(f_0(c))J(f_0,w) \not = 0.
\end{equation}
 
So by the Implicit Function Theorem,  there exists a neighborhood 
$W \subset \mathcal{U}^1$ of $f_0$ in which 
$\mathcal{M}\subset W$, defined by (\ref{MM}), is  a Banach submanifold 
and  so that $W\setminus \mathcal{M}$ has two connected components. 

Let $W_+$ be the connected component containing the maps $f$ satisfying 
$$f^{p_0}(c) > c,$$
and let $W_-$ be  the other component. 
 
We claim  that our assumption that
$f_0$ is stably $\epsilon$-expansive implies that
for every $n$ there exist  neighborhoods $W_n \subset W_{n-1}$ of $f_0$ 
with the following properties:
the sets $W_n\cap W_{+}$ and  $W_n\cap W_{+}$ are  connected and  the critical 
point $c$ of every map
$$f \in W_n\setminus\mathcal{M}$$
is either non periodic or periodic with prime period  $\ge n$. Indeed, 
consider open intervals $I_0$, $I_1, \dots , I_{p_0-1}$, with pairwise disjoint closures, 
such that 
$$|I_i| < \epsilon \mbox{ and } f_0^i(c) \in I_i \mbox{ for every } 0\le i < p_0.$$
Let $W_n$ be a small enough neighbourhood of
$f_0$ so that  every $f \in W_n$ satisfies
\begin{equation}\label{nest}  f^i(c) \subset I_{i \ mod \ p_0},\end{equation}
for every $0\le i\leq n$. In particular, if $f$ has a critical point with prime period 
$p < n$ 
then (\ref{nest}) holds for {\it every} $i$. We claim that $f^{p_0}(c)=c$. 
It is enough to show that 
$$\ell= \# \{f^i(c) \ \colon \ f^i(c) \in I_1     \}=1.$$ 
Define
$$y =\inf \{f^i(c) \ \colon \ f^i(c) \in I_1     \}.$$
Then 
$$f^i[y,f(c)] \subset I_{i+1 \ mod \ p_0} \mbox{ for every } i.$$ 
Since $f$ is $\epsilon$-expansive, this implies that $y=f(c)$, so that
$\ell=1$, as desired.

Consequently, the itinerary of the critical point up to the $n$-th iteration 
is the same for all maps in $W_n\cap W_+$. The  same statement holds for $W_n\cap W_-$. 
In particular there exist sequences
$$\sigma_+ =(\sigma_0^+=C,\sigma_1^+,\sigma_2^+,\dots) \mbox{ and }
\sigma_- =(\sigma_0^-=C,\sigma_1^-,\sigma_2^-,\dots), 
\quad \sigma^+_i,\sigma^-_i \in \{R,L\}, $$
such that the itinerary of the critical point of a 
map  $f \in W_+$ converges to $\sigma_+$ (in the product topology of $\{C,R,L\}^\mathbb{N}$) 
when the map converges to $f_0$, and an analogous statement holds for $\sigma_-$ and $W_-$. 
It is not difficult to see that if $\sigma=(C,\sigma_1,\sigma_2,\dots)$
is the itinerary of the critical point of $f_0$, then $\sigma_i^{+}=\sigma_i^{-}=\sigma_i$ 
if $p_0\ne i$.  

Define
$$ 
C_+:=\sum_{i=0}^{\infty} \frac{1}{[Df_0^{p_0-1}(f_0(c))]^{i}}
\prod_{j=0}^{i}\frac{1}{Df_{0, \sigma_{jp_0}^+}(c)},$$
where we put $Df_{0,R}(c)=\lim_{x\to c, x > c} Df_0(x)$, $Df_{0,L}(c)=\lim_{x\to c, x < c}Df_0(x)$, 
and $Df_{0,C}(c)=1$.  Since $f_0$ is
good there is $\beta >1$ so that
$$|Df_{0,s}(c)\ Df_{0}^{p_0-1}(f_0(c))| > 2 \beta$$
for  $s \in \{L, R\}$, so 
$$ \frac{1}{2\beta} \frac{2\beta-2}{2\beta-1}  \leq C_+  \leq \frac{2\beta}{2\beta-1}.$$

Set
$\lambda:= \inf_{f \in W} \inf_x |Df(x)| > 1$.
For each $f \in W_{p_0m}\setminus \mathcal{M}$ we have
$$
\Big| \prod_{j=0}^{i-1} \frac{1}{[Df^{p_0-1}(f^{jp_0+1}(c))]}\frac{1}{Df(f^{(j+1)p_0}(c))}- 
\frac{1}{[Df_0^{p_0-1}(f_0(c))]^{i}} \prod_{j=0}^{i}\frac{1}{Df_{0, \sigma_{jp_0}^+}(c)}\Big|\\
$$
\begin{equation}\label{rt}
\leq 2 \lambda^{-p_0m}, \quad
\forall 1 \le i \le m ,
\end{equation}
and 
$$\biggl | J(f,v)- \sum_{i=0}^{m} 
\Big[  \prod_{j=0}^{i-1} \frac{1}{[Df^{p_0-1}(f^{jp_0+1}(c))]}
\frac{1}{Df(f^{(j+1)p_0}(c))}\Big] \sum_{\ell=0}^{p_0-1} 
\frac{v(f^{p_0i+\ell}(c)}{Df^{\ell}(f^{p_0i+1}(c))} 
\biggr |$$
$$
\leq \frac{\lambda^{-p_0m} |v|_0}{1-\lambda^{-1}}.
$$
Also, we have
$$|C_+ J(f_0,v)-  \sum_{i=0}^{m} \frac{1}{[Df_0^{p_0-1}(f_0(c))]^{i}}
\prod_{j=0}^{i}\frac{1}{Df_{0, \sigma_{jp_0}^+}(c)} J(f_0,v)|\leq \frac{\lambda^{-p_0m} |v|_0}{1-\lambda^{-1}}.$$

Fix $\delta > 0$ and let $m\ge 1$ be such $2m\lambda^{-p_0m} < \delta$. 
If we assume, as in Claim A, that  $v \in \mathcal{B}^1(I)$,  it is not difficult to see that there is a neighborhood $\widetilde{W}_{\delta,p_0m}$ of $f_0$ such that 
if $f \in \widetilde{W}_{\delta,p_0m}\cap W_+\cap W_{p_0m} $ then
\begin{equation} \label{lip}  | \sum_{\ell=0}^{p_0-1}  \frac{v(f^{p_0i+\ell}(c))}{Df^{\ell}(f^{p_0i+1}(c))} -J(f_0,v)|
\leq \delta |v|_1, \mbox{ for every } 0\le i\leq m .
\end{equation}
Consequently if $f \in \widetilde{W}_{\delta,p_0m}\cap W_+\cap W_{p_0m}$ then
$$|J(f,v)-C_+ J(f_0,v)|\leq \frac{3\delta}{1-\lambda^{-1}}|v|_0 
+ \frac{\delta}{1-\lambda^{-p_0}}|v|_1 .$$
This proves Claim A for $W_+$.

To show Claim B for $W_+$, consider, without loss of generality,  $v \in \mathcal{B}^0(I)$ with $|v|_0\leq 1$. Then we can find $\widetilde{W}_{\delta,p_0m}$ such that  
$$| \sum_{\ell=0}^{p_0-1}  \frac{v(f^{p_0i+\ell}(c))}{Df^{\ell}(f^{p_0i+1}(c))} -J(f_0,v)|
\leq \delta ,\quad \forall 0\le i\leq m,
$$
holds   for every $f \in \widetilde{W}_{\delta,p_0m}\cap W_+\cap W_{p_0m}$. Then
$$|J(f,v)-C_+ J(f_0,v)|\leq \frac{3\delta}{1-\lambda^{-1}}+ \frac{\delta}{1-\lambda^{-p_0}},$$
completing the proof of Claim B for $W_+$.

We can apply a  similar argument to $f \in W_{p_0m}\cap W_-$ and
$$C_-:=\sum_{i=0}^{\infty} \frac{1}{[Df_0^{p_0-1}(f_0(c))]^{i}}
\prod_{j=0}^{i}\frac{1}{Df_{0, \sigma_{jp_0}^-}(c)},
$$ 
with
$ \frac{1}{2\beta} \frac{2\beta-2}{2\beta-1}  \leq C_-  \leq \frac{2\beta}{2\beta-1}$.

The proof of the claims for $f \in \mathcal{M}$ is easier.
\end{proof}

\section{Bifurcations in families of expanding  unimodal  maps}

We are going to see in this section that if a $C^1$ family $f_t$  of good piecewise expanding $C^1$
unimodal  maps is tangent to the distribution of codimension-one subspaces 
$$f  \mapsto \Ker (J(f, \cdot))$$
then there are no bifurcations in this family, that is,  there are homeomorphisms  $h_t$ such
that $h_t\circ f_0= f_t\circ h_t$
for every $t$. In other words, the family is a smooth deformation of $f_0$.
The reverse statement also holds: If $f_t$ is a family such that
$J(f_0,\partial_t f_t|_{t=0})\not = 0$ and if $f_0$ is good,
then there are bifurcations in this family. 


\begin{thm}[Characterization of smooth deformation] \label{equiv}
Let $f_t$, $t \in (-\delta,\delta)$,  be a $C^{k}$ family of  piecewise expanding 
 $C^k$ unimodal maps, with $k\geq 1$.  
Then the following properties are equivalent: 

\begin{itemize} 
\item[A.] For small $t$, the  set of critical relations $R_t$  is constant.
\item[B.] For small $t$, there exists a family $h_t\colon I \to I$  of homeomorphisms so that\\
 $h_t$ is a conjugacy between $f_0$ and $f_t$, $$h_t \circ f_0 = f_t\circ h_t.$$
\item[C.]  For small $t$, there are conjugacies $h_t$, as  in $B$, and we have that  
$$(x,t)\mapsto h_t(x)$$
is continuous and  for each $x \in I$ the function $t \mapsto h_t(x)$  is $C^{k-1 + Lip}$. 
Furthermore   if we restrict $t$ to a compact interval $Q \subset (-\epsilon,\epsilon)$ 
we have that this family is a bounded subset in 
$C^{k-1 + Lip}(Q)$.  (In fact, there is a universal constant $C$ so that
the diameter of this subset is $\le C \sup_{t\in Q} \frac{|f_t|_k}{1-\lambda_{f_t}^{-1}}$.)
\end{itemize}

 Furthermore $A$, $B$ and $C$ imply
\begin{itemize}
\item[D.] For small $t$ we have that $J(f_t,\partial_s f_s|_{s=t})=0$.
\end{itemize}

If we  assume in addition that $f_0$ is stably $\epsilon$-expansive, 
 then D is equivalent to  A, B, and C.
\end{thm}

\begin{proof}
Note that C trivially implies B and A.
\end{proof}

\begin{proof}[Proof of  A implies B]
This implication is a consequence of Milnor-Thurston theory of kneading invariants, but we 
will give a self-contained argument.

Let $\delta>0$ be so that $R_t$ is constant for $t \in (-\delta, \delta)$.
Note that the itinerary $\sigma^{t}$ of the critical point of $f_t$
is constant for $t \in (-\delta, \delta)$. Indeed, if 
$f^i_{t}(c)=c$ for some $i$ and some $t$,
then by definition $(0,i) \in R_t$, and assumption A implies 
$f^i_s(c)=c$ for every small $s$.  By the continuity of the family $f_t$, this
implies that the itinerary 
of $c$  is constant (also if $R_t=\emptyset$).

Let $\mathcal{P}_t$ be the set of points which are either periodic or eventually periodic 
points  of  $f_t$,  and whose forward orbit does not contain the critical point.  It is easy 
to see that $\mathcal{P}_t$ is dense in $I$. We claim that up to taking
a smaller  $\delta >0$,
each point $p \in \mathcal{P}_0$  
has an analytic continuation $h_t(p)$, defined for every $|t|< \delta$. Moreover 
$$h_t \colon \mathcal{P}_0 \to  \mathcal{P}_t$$
is a bijection. In fact, since the forward orbit of $p$ does not contain the critical point, 
we can find a maximal open interval $Q$, where  the analytic continuation $h_t(p)$ is 
(uniquely) defined. If there exists 
$t_{\infty} \in \partial Q \cap (-\delta,\delta) $,  choose    $t_n\in Q$, with 
$lim_{n \to \infty} t_n= t_{\infty}$.

Since $Q$ is maximal, every  accumulation point $q$ of  the sequence $h_{t_n}(p)$ has a 
priori  the itinerary of $p$,  replacing at least one of its symbols by $C$.  But note that  
since the $f_{t}$ are  piecewise expanding, and since we proved that the itinerary of the 
critical point under
$f_t$ is constant,   every  itinerary  obtained by replacing $C$ 
by either $R$ or  $L$ symbols  in the itinerary of the critical point is forbidden for $p$. 
So $\partial Q\cap(-\delta,\delta)  =\emptyset$ and $h_t(p)$ is defined for every  $t$.  
Of course $h_t(p) \in \mathcal{P}_t$. Furthermore $h_t$ is injective, since $h_t(p)$ has 
the same itinerary as $p$ and distinct points in $\mathcal{P}_0$ have distinct itineraries. 

It remains to prove  that $h_{t_0}(\mathcal{P}_0)= \mathcal{P}_{t_0}$, for every $t_0$. 
This can be achieved by considering a smooth re-parametrization $g_u$ of the family $f_t$ 
such that $g_0=f_{t_0}$ and  and applying  the argument above  to construct $h^{-1}_{t_0}$. 

Due the uniqueness of the analytic  continuation 
\begin{equation} \label{conj} h_t \circ f =  f_t \circ h_t \end{equation}
on $\mathcal{P}_0$.
Moreover $p < q$ implies $h_t(p) < h_t(q)$ for every $t$. By the density of 
$\mathcal{P}_t$, for every $t$, we can extend $h_t$ to a homeomorphism 
$h_t \colon I \to I$.
The continuity of $h_t$ and Eq. (\ref{conj}) imply that $f_0$ is conjugate to  
$f_t$ by $h_t$.
\end{proof}

\begin{proof}[Proof of B implies C]
See \cite[Proposition 2.4]{BS} (the proof  there works for $k\ge 1$).
\end{proof}

 \begin{proof}[Proof of  A, B, C implies D] It is enough to show that A. implies D. 
First, suppose that $R_t\neq \emptyset$. Then $f_t$ has a periodic critical point with prime period $p$, for all small $t$, that is,
$f_t^{p-1}(f_t(c))=c$
for small $t$. Differentiating with respect to $t$, we obtain
$$\partial_t f_t^{p-1}(f_t(c)) + Df^{p-1}_t(f_t(c))\ \partial_t f_t(c)=0 .$$
So (using (\ref{clarif1}) for $f_t$ and $\partial_t f_t$)
$$J(f_t,\partial_t f_t)= \frac{\partial_t (f_t^{p-1}\circ f_t)(c)}{ Df^{p-1}_t(f_t(c))}=
 \frac{\partial_t f_t^{p-1}(f_t(c))}{ Df^{p-1}_t(f_t(c))}+ \partial_t f_t(c)=0.$$

Now assume that  $R_t=\emptyset$ for small $t$ and suppose for
a contradiction that $J(f_{t_0},\partial_{t} f_t|_{t=t_0})\neq 0$ for some  small $t_0$. 
By Proposition \ref{kerno}, Claim B., either $J(f_t,\partial_t f_t) \geq \xi >0$ for every $t$ close to $t_0$, or $J(f_t,\partial_t f_t) \leq \xi < 0$ for every $t$ close to $t_0$. Without loss of generality, assume the first case. Using  (\ref{clarif2}) for $f_t$ and $\partial_t f_t$, 
and the fact that $\theta= \inf_{t,x} |Df_t(x)|>1$, we find
$\delta > 0$ and $k_0\ge 1$ so that 
$$\frac{\partial_t f^{k}_t(c)}{D f^{k-1}_t(f_t(c))}\geq \frac{\xi}{2},
\quad \forall |t-t_0|\le \delta , \forall k \ge k_0.$$
So, 
$$2\geq | f^{k}_{t_0+\delta}(c)-f^{k}_{t_0}(c)|\geq \frac{\xi}{2} \theta^{k-1}\delta,$$
for every $k \ge k_0$, which is absurd since $\theta >1$.
\end{proof}

\begin{proof}[Proof of  D implies A]
We assume  stable
$\epsilon$-expansivity of $f_0$. Consider the set of uniformly bounded functions
$$c_n \colon  \{ t \colon \ |t|< \delta\} \to I,$$
with $c_n(t)=f^n_t(c)$. We claim that this family is equicontinuous.

Write
$v_t = \partial_s f_s|_{s=t}$. By Proposition $\ref{stce}$, there exists
for each $t$  a unique bounded 
function $\alpha_t\colon I \to \mathbb{R}$ satisfying $\alpha_t(c)=0$ and
\begin{equation} \label{nocrit} v_t(x)=\alpha_t(f_t(x))-Df_t(x)\alpha_t(x)\end{equation}
for every $x\not=c$. In addition, since we assumed
$J(f_t,v_t)=0$, we have
\begin{equation}\label{crit} v_t(c)=\alpha_t(f_t(c))\, , \forall t .\end{equation}

Consider a solution
$g\colon \{ t \colon \ |t|< \delta\} \to I$
of the differential equation
\begin{equation}\label{ode} \frac{dg}{dt}(s) = \alpha_t(g(s)).\end{equation}
We claim that the function
$G \colon \{ t \colon \ |t|< \delta\} \to I$
defined by 
$$G(t)=f_t(g(t))$$
is also a solution of (\ref{ode}).  Indeed, as a consequence of (\ref{nocrit}), if 
$g(t_0)\not=c$, we have
\begin{align*}
\frac{dG}{dt}(t_0) &= v_{t_0}(g(t_0))+Df_{t_0}(g(t_0))\frac{dg}{dt}(t_0)\\
&= v_{t_0}(g(t_0))+Df_{t_0}(g(t_0))\alpha_{t_0}(g(t_0))= \alpha_{t_0}(f_{t_0}(g(t_0)))= 
\alpha_{t_0}(G(t_0)).
\end{align*}
If $g(t_0)=c$ then by  (\ref{ode}) we have
$\frac{dg}{dt}(t_0)=0$. Therefore, using that
the $f_t$ are piecewise uniformly Lipschitz and the
family is $C^ 1$, we get
\begin{align*}
\frac{G(t_0+\eta)-G(t_0)}{\eta} &= \frac{f_{t_0+\eta}(g(t_0+\eta))- 
f_{t_0+\eta}(g(t_0))}{\eta} +  
\frac{f_{t_0+\eta}(g(t_0))-f_{t_0}(g(t_0))}{\eta}\\
&= \frac{O(|g(t_0+\eta)-g(t_0)|)}{\eta}  +   v_{t_0}(g(t_0)) + o(h)= v_{t_0}(g(t_0))+ o(\eta).
\end{align*}
So 
$$\frac{dG}{dt}(t_0)= v_{t_0}(g(t_0)) = v_{t_0}(c)=\alpha_{t_0}(f_{t_0}(c))= 
\alpha_{t_0}(f_{t_0}(gt_0))).$$
Consequently $g_n(t)=f_t^n(gt))$ is a solution of (\ref{ode}), for every $n$. 

Of course the constant function 
$c_0(t)=c$
is a solution of (\ref{ode}). Since the functions $\alpha_t$ can be uniformly bounded by a 
constant which is independent of $t$, we conclude by (\ref{ode}) that the set of functions
$c_n(\cdot)$ is equicontinuous.

Suppose now that there is $t_0$ so that
$c$ is a periodic point of $f_{t_0}$.  
If the prime period of $c$ is $p$, choose  open intervals 
$I_0, I_1, \dots I_{p-1}$,
with pairwise disjoint closures, $|I_i| < \epsilon$, 
and  such that 
$$f^i_{t_0}(c) \in I_{i \ \mod \ p}\quad \forall i .$$
Since $\{ c_n(\cdot) \}$ is an  equicontinuous set of functions, there exists 
$\delta_0 > 0$ such that 
\begin{equation}\label{renor}  f^i_{t}(c) \in I_{i \ \mod \ p}\end{equation}
for every $i$ and every $t$ such that $|t-t_0|<\delta_0$. We claim that if $|t-t_0|<\delta_0$ 
then the map  $f_t$ has a periodic critical point with the same itinerary
as that of $c$ for $f_{t_0}$.  By (\ref{renor}), it is enough to show that 
$$N= \# \{f^i_t(c)\colon i \ \mod \ p=1\}=1.$$ 
Define
$$y = \inf  \{f^i_t(c)\colon i \ \mod \ p=1\}.$$
The definition of $y$ implies
$$f^{p}_t[y,f_t(c)]\subset [y,f_t(c)],$$
and 
$$f^i_t[y,f_t(c)] \subset I_{i+1 \ \mod \ p}$$
for every $p$. This implies  
$|f^i_t([y,f_t(c)])|< \epsilon$ for every $i$. By the stable $\epsilon$-expansivity 
of $f_0$ we must have $y=f_t(c)$ which implies $N=1$.

So we conclude that for every itinerary $\sigma$ of length $p$, the set of parameters $\mathcal{O}$ 
such that 
$f_t$ has  a $p$-periodic critical point with itinerary $\sigma$ is an open set.  Of course for 
all parameters in the closure of $\mathcal{O}$, $f_t$  has a $p$-periodic critical point, 
but,  {\it a priori,} not with prime period $p$.  But if we apply the same argument to this 
boundary parameter, we conclude that its critical point has the same itinerary as points in 
$\mathcal{O}$. 
This implies that either $\mathcal{O}=\emptyset$ or $\mathcal{O}=\{ t \colon \ |t|< \delta\}$.
If $\mathcal{O}=\emptyset$ for each finite orbit, then each $f_t$ has an infinite
postcritical orbit and an empty $R_t$.
So the set of critical relations $R_t$ does not depend on $t$.
\end{proof}

We mention an easy consequence of Theorem~\ref{equiv} which will be useful
in the proof of Theorem~\ref{create}:

\begin{cor}[Unstable families]\label{unstable} Let $f_t$ be a $C^1$ family of 
 piecewise expanding $C^1$ unimodal  maps such  that $f_0$ is good and
$$J(f_0,\partial_t f_t|_{t=0})\not =0.$$
\begin{itemize}
\item[A.] If $f_0$ has a periodic critical point  then there exists a sequence of parameters 
$t_n\to 0$
such that the critical point of $f_{t_n}$ is not periodic.
\item[B.]  If $f_0$ has a non periodic critical point then there exists a sequence of 
parameters 
$t_n\to 0$
such that the critical point of $f_{t_n}$ is periodic.
\end{itemize}
\end{cor}

\begin{proof}[Proof of Claim A] By Corollary  \ref{continuity} 
and the 
continuity of $t\mapsto \partial_t f_t$ in the $\mathcal{B}^0(I)$ norm, there exists $\tilde{\delta} > 0$ such that 
\begin{equation}\label{nzero}  J(f_t,\partial_s f_s|_{s=t}) \not =0,
\quad \forall |t|\leq {\delta_0}.
\end{equation}
Suppose by contradiction
that for all parameters $|t|\leq \delta_1 < {\delta_0}$ the critical point of $f_t$ 
is  periodic. Define
$$P_n := \{t\colon \  f_t^n(c)=c \ \mbox{and} \  |t| \leq \delta_1    \}.$$
Of course  $P_n$ is closed. By the Baire Theorem there exists 
$n_0 \ge 1$ so that $P_{n_0}$ contains a nonempty connected open 
set $Q \subset P_{n_0}$.  For each $1\le i \leq n_0$, let $P'_i \subset P_{n_0}$ 
be set of parameters for which $f_t$ has a critical point whose prime period is  equal  or 
larger than $i$.  Of course each $P'_i$ is an open subset of $P_{n_0}$. Let
$$p = \max \{i\colon \ P'_i\not = \emptyset  \}.$$
Then there exists an open set $U$ such that the critical point of $f_t$ 
has prime 
period $p$ if $t\in U$.  In particular the set of critical relations $R_t$ is constant on $U$. 
By the implication $A \Rightarrow D$ in Theorem \ref{equiv}, 
$J(f_t,\partial_t f_t)=0$
for every $t \in U$, which contradicts (\ref{nzero}).
\end{proof}

\begin{proof}[Proof of Claim B]  The proof in this case is even easier. By Proposition
\ref{kerno}~ B., we have (\ref{nzero}) for some $\delta_0>0$. 
If there are non periodic critical points for $f_t$ for all small enough $t$, 
then the set of critical relations $R_t$ is empty for those $t$. 
By $A \Rightarrow D$ in Theorem \ref{equiv}, 
$J(f_t,\partial_t f_t)=0$
for all small enough $t$,  which contradicts (\ref{nzero}).
\end{proof}

\section{Finding or approximating  families tangent to a given  horizontal direction}

We can now state and prove our second main result:

\begin{thm}  \label{create}Let $k\ge 2$, let $f$ be a good piecewise expanding $C^k$
unimodal map, and let  $v, w \in \BB^{k}(I)$ 
satisfy  $v(-1)=v(1)=w(-1)=w(1)=0$,
$J(f,v)=0$ and $J(f,w)\neq 0$.  Then for every $C^k$
family $f_t$ of  piecewise  expanding  $C^k$ unimodal maps such 
that  $f_0=f$ and $\partial_t f_t|_{t=0}=v$,
there exists $\delta > 0$ and  a {\it unique}  
continuous function $b:(-\delta,\delta)\to \real$, such that
$b(0)=0$ and  that
$$\tilde f_t = f_t + b(t)w$$ 
 is topologically conjugate with $f$ for all $|t|<\delta$.

Furthermore this unique function $b$ is in fact $C^{k-1+Lip}$ and
satisfies $b'(0)=0$ (in particular $\partial \tilde f_t|_{t=0}=v$),
and the family $\tilde f_t$ is a $C^{k-1+Lip}$-family of 
piecewise expanding $C^k$ unimodal maps.
 
In addition, there exists a sequence of $C^{k}$ families of
piecewise  expanding  $C^k$ unimodal maps
$t  \mapsto g_{t,n} $ ($t\in (-\delta,\delta)$)
such that 
\begin{itemize} 
\item[-] the map $g_{t,n}$ is topologically conjugate with $g_{0,n}$, for each $t$ and $n$,
\item[-]  the critical point of $g_{0,n}$ is periodic for each $n$,
\item[-]  For each $t$ the map  $g_{t,n}$ converges to the  map 
 $\tilde{f}_{t}$ in 
the $\BB^{k-1}(I)$ topology.
\end{itemize}
\end{thm}

\begin{proof}[Proof of the existence of $\tilde f_t$] 
Note that if $f_0(c)=+1$, then either $f_t(c)=+1$ for all small enough $t$
(in which case we may take $b(t)\equiv 0$, so that existence of $\tilde f_t$
is proved) or $f_t(c) <1$ for all nonzero small enough $t$.
Denote $v_t = \partial_s f_s|_{s=t}$. For small $\eta >0$
set $M_\eta=\{|t|< \eta, \ |\theta|< \eta \}$ and consider
\begin{equation}\label{fth}
f_{(t,\theta)}=f_t + \theta w , \quad
(t,\theta)\in M_\eta   .
\end{equation}
(In fact, if $f_0(c)=+1$ but $f_t(c)<1$ for all small nonzero $t$ we must
 take $M_\eta=\{|t|<\eta, |\theta|<\Theta(t)\}$ with
 $\Theta$ a $C^ 1$ function so that $\Theta(0)=0$, $\Theta(t)>0$
for $t\ne 0$.)
Since $J(f_0,v_0)=0$ but $J(f_0,w)\ne 0$,
Propositions \ref{kerperiodic} and \ref{kerno} and Corollary ~ \ref{continuity} imply that if  
$\eta$ is small enough then for all  $(t,\theta)\in M_\eta$
the linear space
$$ 
\mathbb L(t,\theta)=\{ (\alpha ,\beta) \in \real^2\ \colon 
J(f_{(t,\theta)},\alpha v_t +\beta w)=0   \}
$$ 
is a one-dimensional subspace of $\real^2$, which depends continuously
on $(t,\theta)$ and never coincides
with the vertical line $\{0\}\times \real$. In other words:
There exists a uniquely defined  function $d:M_\eta \to \real$ 
so that
\begin{equation}\label{d()}
v_t+ d(t,\theta)w \in \Ker (J(f_{(t,\theta)},\cdot)) .
\end{equation}
In addition, $d(0,0)=0$ and $d$ is continuous.
Consider the $C^1$-integral curve $b$ of the 
ordinary differential  equation
$$\frac{db}{dt} = d(t,b(t)),\quad  b(0)=0.
$$
Since $d(0,0)=0$, if $\eta$ is small, then the solution $b$ is defined for 
$|t| < \eta$.  
As a consequence,   the family 
$ \tilde f_{t}=f_{(t,b(t))}= f_t + b(t)w$
satisfies
\begin{equation*}
J(\tilde f_{t},\partial_s \tilde f_{s}|_{s=t})= 0 ,\quad
\forall |t|<\eta .
\end{equation*}
By $D\Rightarrow B$ in Theorem~\ref{equiv},  $\tilde f_t$
is topologically conjugate with $f_0$ for small  $t$.
\end{proof}

\begin{proof}[Proof of the uniqueness of $\tilde f_t$.] Suppose that 
$\underline b$, $\bar b$ are two continuous functions with
$\underline b(0)=\bar b (0)=0$ and such that both maps
\begin{equation} \label{tf} f_t +\underline  b(t)w,\quad  f_t + \bar b(t)w,\end{equation}
are topologically conjugate to $f$ for each small $|t|<\delta$.
Using the map $d:M_\eta\to \real$ from (\ref{d()}) in the proof of the existence of 
$\tilde f_t$, choose $0<\hat{\eta} < \tilde{\eta} < \eta$ such that if $(t_0,\theta_0)
\in M_{ \hat{\eta}}$ then the ordinary differential equation 
\begin{equation}\label{ode2} \frac{d b(t)}{dt} = d(t,b(t))\end{equation}
with initial condition $b(t_0)= \theta_0$, has a $C^1$-solution defined for every 
$|t|< \tilde{\eta}$, and, moreover, we have  $|b(t)|< \tilde{\eta}$ for $|t|< \tilde{\eta}$. 
Such $\hat{\eta}$, $\tilde{\eta}$ exist, since $d(0,0)=0$.

Suppose there is $|t_0|< \hat{\eta}$ such that $\underline b(t_0)\neq \bar b(t_0)$.  
Since $\bar b$ and $\underline b$ are continuous, up to taking a smaller
$t_0$ we may assume that $\max(|\underline b(t_0)|,| \bar b(t_0)|)< \hat \eta$.
To fix ideas, assume $0\le \underline b(t_0) <\bar b(t_0)$ (the other cases are
similar).
Then for every 
$\theta_0 \in (\underline b(t_0), \bar b(t_0))$
we can find a solution $$ b\colon (-\tilde \eta, \tilde \eta)\to \mathbb{R}$$ 
for
(\ref{ode2})  such that $b(t_0)=\theta_0$. 
By the Intermediate Value Theorem, there exists $t_1 \in [0,t_0)$ such that 
$b(t_1)=\underline b(t_1)$. 

Note that by (\ref{ode2}) and the definition of $d(\cdot,\cdot)$
\begin{equation}\label{ift} J(f_t+ b(t)w, v_t+  b'(t)w)=0 .\end{equation}
Thus, by $D\Rightarrow B$ in Theorem \ref{equiv}  for $b$, and by assumption
for $\underline b$, $\bar b$, 
all maps 
$$f_t + b(t)w,\quad f_t+\underline b(t)w,\quad f_t+\bar b(t)w, 
\quad |t|<\hat \eta
$$
are in the  topological class of $f$. 
As a consequence, for every 
$\theta\in  (\underline b(t_0), \bar b(t_0))$,
the map $f_{t_0}+\theta w$ is topologically conjugate with $f$. Consequently there is no change of combinatorics in the $C^\infty$ family 
of piecewise expanding $C^k$ unimodal maps
$$
s\mapsto f_{t_0} + (\theta_0+s)w,
\quad |s| < \eta(\theta_0),
$$
and $B \Rightarrow D$ in Theorem~\ref{equiv} gives $J(f_{t_0}+\theta_0 w,w)=0$.

Taking a sequence $t_n\to 0$
such that $\underline b(t_n)\neq \bar b(t_n)$, the argument above
gives a sequence $\theta_n\to 0$
so that $J(f_{t_n}+\theta_n w,w)=0$, and
Corollary~\ref{continuity} implies
$J(f_{0},w)=0$, contradicting the assumption on $w$. 
\end{proof}

\begin{proof}[Proof of the  $C^{k-1+Lip}$ regularity, construction
of $g_{t,n}$]
Recall  (\ref{fth}) and the characterization (\ref{d()}) of $d(\cdot,\cdot)$.
By Corollary \ref{unstable}.B, there exists a sequence $\theta_n\to 0$ such that 
$f_{(0,\theta_n)}$ has a periodic critical point (if $f_0$ has a periodic critical point, 
define $\theta_n=0$, for every $n$).  Consider the $C^1$-integral curves $b_n$ of the 
ordinary differential  equation
\begin{equation}
\label{oden}
\frac{db_n}{dt} = d(t,b_n(t)),\quad  b_n(0)=\theta_n.
\end{equation}

Note that since $d(0,0)=0$, if $\eta$ is small, then the solution $b_n$ is defined for 
$|t| < \eta$, provided $n$ is large enough.  
As a consequence,  for all large enough  $n$, 
\begin{equation}\label{family} g_{t,n}=f_{(t,b_n(t))}= f_t + b_n(t)w\end{equation}
satisfies
\begin{equation}\label{inker}
J(g_{t,n},\partial_s g_{s,n}|_{s=t})= 0 ,\quad
\forall |t|<\eta .
\end{equation}

Let $p_n$ be the prime period of the turning point of $g_{n,0}= f_{(0,\theta_n)}$.  
By (\ref{inker}) and $D \Rightarrow B$ in Theorem \ref{equiv} 
we have that $g_{t,n}$ is topologically conjugate with $g_{n,0}$, 
so $g_{t,n}$ has a critical point with the same prime period $p_n$. 

We shall first prove that each $g_{t,n}\in \UU^{k}(I)$, by showing that
each function $b_n$ is $C^k$.
Indeed consider the non-linear functional 
 $$F_n(t,\theta)= f_{(t,\theta)}^{p_n}(c) - c.$$
Then $F_n$ is $C^{k}$  on $M_\eta$,
and  if $f_{(t,\theta)}$ has a periodic point with prime period $p_n$ our
assumption on $w$ gives  
(recalling (\ref{clarif}), as for
(\ref{Imp1}))
$$\partial_\theta F_n(t,\theta) = 
Df^{p_n-1}_{(t,\theta)}(f_{(t,\theta)}(c)) J(f_{(t,\theta)}, w) \not = 0.$$
Since
$F_n(t,b_n(t))=0$, 
the Implicit  Function Theorem implies that $t\mapsto b_n(t)$ is $C^{k}$. 
For further use, 
note also that
$$\partial_t F_n(t,b_n(t)) + \partial_\theta F(t,b_n(t))b_n'(t)=0,$$
and since
$\partial_t F_n(t,\theta) = Df^{p_n-1}_{(t,\theta)}(f_{(t,\theta)}(c)) J(g_{t,n},v_t) $,
we obtain
\begin{equation} \label{quotient} b_n' (t)= -\frac{J(g_{t,n},v_t)}{J(g_{t,n},w)}.
\end{equation}

We shall next show that  the families
$\{(t,x)\mapsto g_{t,n}(x), \quad n \in \mathbb{N}\}$
form a bounded subset of $C^{2}$
(in the sense of families (\ref{families})) as the first step in the
inductive proof that this set is bounded for $C^{k}$.  
Let $\lambda=\inf_{t,n} \lambda_{g_{t,n}}$. We have $\lambda > 1$.

Since $f_t$ and $w$ are  in $\BB^{2}(I)$,   
and $b_n$, $b'_n$ are uniformly  bounded in $n$
(use (\ref{oden})),  there exist  
by the definition (\ref{family})
uniform  upper bounds for the derivatives
 $$\partial_t \ g_{t,n},\ \partial_x \ g_{t,n} , \  \partial^2_x \ g_{t,n},\  
\partial^2_{xt} \ g_{t,n},\ \partial^2_{tx}\  g_{t,n}.$$
So we must only estimate 
$\partial^2_{tt} \  g_{t,n}$. By (\ref{family}),  it is enough to show that 
 $\{ b_n\}_n$ is a bounded subset in $C^{2}$. Since we already
 know that each $b_n$ is $C^2$, it is enough to get an $n$-uniform bound
on the Lipschitz constant of $b_n'$.
In view of  (\ref{quotient}),  we first show that for every vector field $u\in \BB^1(I)$, 
the map
$t\mapsto J(g_{t,n},u)$
is Lipschitz, and its  Lipschitz norm does not depend on $n$.
By  $B \Rightarrow C$  in Theorem ~\ref{equiv}, the conjugacies 
$h_{t,n}\circ g_{0,n} = g_{t,n}\circ h_{t,n}$
are such that (the $C^{k-1+Lip}$ map) $t \mapsto h_{t,n}(x)$
is $K$-Lipschitz, with
$$
K \leq  \tilde{C} \ \frac{\sup_{n, |t|< \eta} |v_t+b'_n(t)w|_0}{1-\lambda^{-1}}
.
$$
This implies in particular that $t\mapsto u(g^i_{t,n}(c))$ is Lipschitz
uniformly in $t$, $i$ and $n$.

So (we omit $n$ in $g_{t,n}$  and $h_{t,n}$to avoid a cumbersome notation) 
\begin{align} \label{Jesti} &|J(g_{t + \delta},u)- J(g_{t},u)|
\leq  \sum_{i =0}^{p_n-1}\Big| \frac{u(g^i_{t+\delta}(c))}{Dg_{t+\delta}^i(g_{t+\delta} (c))}
-\frac{u(g^i_t(c))}{Dg_t^i(g_t(c))}\Big|\\
\nonumber & \quad\leq \sum_{i =0}^{p_n-1}\frac{|u(h_{t+\delta}(g^i_0(c))-u(h_{t}(g^i_0(c))|}
{|Dg_{t+\delta}^i(g_t(c))|}
\\ \nonumber &\qquad\qquad
+   \sum_{i =0}^{p_n-1}|u(g^i_t(c))|\Big| 
\frac{1}{Dg_{t+\delta}^i(g_{t+\delta} (c))}-\frac{1}{Dg_t^i(g_t(c))}\Big|.
\end{align}

Note that   
\begin{align}\label{dstep}
&\Big| \frac{1}{Dg_{t+\delta}^i(g_t(c))}-\frac{1}{Dg_t^i(g_t(c))}\Big| \\ 
\nonumber &
\leq  \sum_{j=0}^{i-1} \frac{1}{\lambda^{i-1}} \Big|   
\frac{1}{Dg_{t+\delta}(h_{t+\delta}(g_0^{j+1}(c)))}-\frac{1}{Dg_t(h_{t}(g_0^{j+1}(c)))}\Big| \\
\nonumber &\leq \frac{1}{\lambda^{i-1}} \biggl (\sum_{j=0}^{i-1} \Big|   
\frac{1}{Dg_{t+\delta}(h_{t+\delta}(g_0^{j+1}(c)))}
-\frac{1}{Dg_t(h_{t+\delta}(g_0^{j+1}(c)))}\Big| \\
\nonumber &\qquad\qquad+ 
 \Big|   \frac{1}{Dg_{t}(h_{t+\delta}(g_0^{j+1}(c)))}-\frac{1}{Dg_t(h_{t}(g_0^{j+1}(c)))}\Big|
\biggr )
\leq C\frac{i}{\lambda^i} \delta.
\end{align}
In the last step we used that  $\partial^2_{tx} \ g_{t,n}$ is  bounded uniformly in $n$.
So  (\ref{Jesti}) gives
$$
|J(g_{t + \delta,n},u)- J(g_{t,n},u)|\leq\sum_{i =0}^{ p_n-1} 
\Big[  \frac{C|u|_1\delta}{\lambda^i} + \frac{|u|_0i\delta}{\lambda^i}\Big]
\leq C|u|_1\delta,$$
which proves
that $t\mapsto J(g_{t,n},u)$ is $C|u|_1$-Lipschitz, 
uniformly in $n$.

Next,
we obtain  from (\ref{quotient}) that
\begin{align*}
|b'_n(t+\delta)- b'_n(t)|
&\leq \frac{|J(g_{t+\delta},v_{t+\delta})-J(g_{t+\delta},v_{t})| +
|J(g_{t+\delta},v_{t})-J(g_{t},v_t)|}{|J(g_{t+\delta},w)|} 
\\
&\quad + |J(g_{t},v_t)|\Big| \frac{1}{J(g_{t+\delta},w)}-\frac{1}{J(g_{t},w)}     \Big|
\leq K\delta \max(\sup_t |v_t|_1,|w|_1).
\end{align*}
We used that $J(g,\cdot)$ is linear and that $J(g_{t,n},w)$ is bounded away from zero and infinity, uniformly in $n$
and  $|t| < \eta$ (by Propositions~\ref{kerno} and \ref{kerperiodic} since $g_{t,n}$ and
$f_0$ are $\BB^ 2(I)$-close, using that $\sup_{n,t}|b_n(t)|<\infty$).

This proves our
claim that $b_n$ is $C^{1+Lip}$, and thus $C^2$, uniformly
in $n$, and thus the uniform $C^2$ claim on $g_{t,n}$. 
Note that  by $B \Rightarrow C$ in Theorem~\ref{equiv} the maps
$t\mapsto h_{t,n}(x)$ are $C^{1+Lip}$ uniformly in $n$ and $x$.
  If $k=2$, we are done.
If $k\ge 3$,  we have concluded
the first step in the induction.

Before we perform the inductive step, we introduce some notation
and terminology. Let $X$ be a set and let $f_\lambda$, $\lambda \in \Lambda$, be an indexed family  $\mathcal{F}$ of  functions on $X$.  For each formal {\it monomial}  $\lambda_1\lambda_2\dots\lambda_n$, we can associate the function $f_{\lambda_1} f_{\lambda_2}\dots f_{\lambda_n}$. This function is called a $\mathcal{F}$-monomial combination of degree $n$. A {\it $\mathcal{F}$-polynomial} combination of degree $n$ is a finite sum of $\mathcal{F}$-monomials whose maximal degree is $n$.   Note that if $\Lambda_1$ is a finite subset of $\Lambda$  and $P_n \in \mathbb{N}_n[\Lambda_1]$ is a polynomial with non  negative integer coefficients, we can associate to it a $\mathcal{F}$-polynomial combination of degree $n$. We will call this combination the  $P$-combination of the family $\mathcal{F}$.

If $\mathcal{F}_1$ and $\mathcal{F}_2$ are  families  indexed by $\Lambda_1$ and $\Lambda_2$, we will denote by $\mathcal{F}_1\cup \mathcal{F}_2$ the disjoint union of these  families, indexed by the disjoint union $\Lambda_1$ and $\Lambda_2$.

 If $X$ is an open interval,  all functions in $\mathcal{F}$ are differentiable and $\mathcal{F}'$ is the indexed family of derivatives of functions in $\mathcal{F}$,  then  the derivative of  a sum of $m$ $S$-monomials of degree $n$ is a sum of $m\cdot n$ $\mathcal{F}\cup \mathcal{F}'$-monomials of degree $n$.

Suppose by induction that $\{ b_n\}_n$ is a bounded family in $C^{q}$, with $2\leq q<  k$. Then 
$g_{t,n}$ is a $C^q$ family.  By  $B \Rightarrow C$  in Theorem ~\ref{equiv}, the conjugacies 
$h_{t,n}\circ g_{0,n} = g_{t,n}\circ h_{t,n}$
are such that $t\mapsto h_{t,n}(x) \in C^{q-1+Lip}$,
uniformly in $n$ and $x$, and,  setting
$$\mathcal{C}_{n,i}^{q-1}=\{ \partial_t^j h_{t,n}(g_{0,n}^\ell(c))\colon  \ 1\leq  j\leq q-1,   \ell\leq i  \},$$
then
\begin{equation}\label{familyc}\cup_n  \mathcal{C}_{n,p_n-1}^{q-1}\end{equation}
is a bounded family in $C^{Lip}$. Note also that for every $x\neq c$
$$\partial_x^a\partial_t^b g_{t,n}(x)=\partial_t^b  \partial_x^a g_{t,n}(x)= \partial_x^a f_t (x)+ \partial_t^b b_n(t)\  \partial_x^a w(x)$$
for every $a\leq k$ and $b\leq q$. In particular for each $i< p_n$, if we define the indexed families of functions  
$$\mathcal{G}^q_{n,i}:=\{\partial_x^a\partial_t^b g_{t,n} (h_{t,n}(g_{0,n}^\ell(c)))\colon \ b< q, a\leq q, 1\leq \ell \leq i \}$$
then 
\begin{equation}\label{familyg} \cup_n \mathcal{G}^q_{n,p_n-1}\end{equation}
is also a bounded subset in $C^{Lip}$.  It is easy to show by induction on $q$ that  there exist $C_q$ and $d_q$ such that for each $1\leq \ell< p_n$,  there exists a $P_{q,\ell}$-combination $\psi_{n,\ell,q}$ of the family $\mathcal{C}_{n,\ell}^{q-1} \cup \mathcal{G}^q_{n,\ell}$  such that
$$\partial^{q-1}_t \frac{1}{Dg_{t,n}(h_t(g_{0,n}^\ell(c)))}= \frac{\psi_{n,\ell,q}}{(Dg_{t,n}(h_t(g_{0,n}^\ell(c)))^{2^{q-1}}},$$
and $P_{q,\ell}$ is a sum of at most $C_q$ monomials of maximal degree $d_q$. In particular if we define the indexed families
$$\mathcal{F}^{q-1}_{n,i}:= \{  \partial^r_t \frac{1}{Dg_{t,n}(h_t(g_{0,n}^\ell(c)))}\colon \  0\leq r\leq q-1, 1\leq \ell< p_n   \}$$
then 
\begin{equation}\label{familyf}\cup_n \mathcal{F}^{q-1}_{n,p_n-1}\end{equation}
 is a bounded set in $C^{Lip}$.

Let $u \in \mathcal{B}^k(I)$. We claim that 
$\{  J(g_{t,n},u)\}_n$
is a bounded subset of functions in $C^{q-1+Lip}$.  Indeed, for each $n$ and $i< p_n$, define the indexed families of functions
$$\mathcal{O}_{n,i}^{q-1}=\{  D_x^{j} u (h_{t,n}(g_{0,n}^\ell(c)))\colon \   0\leq j \leq q-1, 1\leq \ell\leq i\},$$                                                                                                      
Of course
\begin{equation}\label{familyo}\cup_n \mathcal{O}_{n,p_n-1}^{q-1}\end{equation}
is a bounded subset of functions in $C^{Lip}$.

Note that for each $i < p_n$ 

$$u_{0,i,n}:= \frac{u(g^i_{t,n}(c))}{Dg_{t,n}^i(g_{t,n}(c))}=u(h_{t,n}(g^i_{0,n}(c))) \prod_{\ell=1}^{i} \frac{1}{Dg_{t,n}(h_{t,n}(g_{0,n}^{\ell}(c)))},$$
in particular  there exists a monomial $P_{0,i}$ of degree $i+1$ such that $u_{0,i,n}$ is a $P_{0,i}$-combination of the family $\mathcal{C}_{n,i}^0\cup \mathcal{F}_{n,i}^0\cup \mathcal{O}_{n,i}^0$. It can be easily proven by induction on $q$ that for every $1\leq i < p_n$  there exists a polynomial $P_{q,i}$ such that 
$$\partial_t^{q-1} \ \frac{u(g^i_{t,n}(c))}{Dg_{t,n}^i(g_{t,n}(c))}$$
is a $P_{q,i}$-combination of the family $\mathcal{C}_{n,i}^{q-1}\cup \mathcal{F}_{n,i}^{q-1}\cup \mathcal{O}_{n,i}^{q-1}$. Furthermore $P_{q,i}$ is the sum of at most  $(i+q-1)!/i!$  monomials with maximal degree $i+q$, and each monomial contains at least $\max\{0, i-q+1\}$ indexes in 
$\mathcal{F}^{0}_{n,i}$. So if 
$$\bigcup_n  \mathcal{C}_{n,p_n-1}^{q-1}\cup \mathcal{F}_{n,p_n-1}^{q-1}\cup \mathcal{O}_{n,p_n-1}^{q-1}.$$
belongs to the ball of radius $R$ in $C^{Lip}$, it is easy to see that

\begin{equation}|\partial_t^{q-1} J(g_{t,n},u)|_{Lip}\leq 
\sum_{j=0}^{q-1} \frac{(j+q)!}{j!} R^{j+1+q} +\sum_{i\geq q}^{\infty} \frac{(i+q)!}{i!} \frac{R^q }{\lambda^{i-q}} .
\end{equation} 
Thus, by (\ref{quotient}),   $b_n$ belongs to a bounded subset in $C^{q+Lip}$, and
thus in $C^{q+1}$.  That concludes the induction step. In particular  $\{g_{t,n}(x), n \}$ belongs to a bounded set of 
$C^{k}$ families (in the sense of families (\ref{families})).

We can thus choose a convergent subsequence
$\lim_{j\to \infty} g_{n_j,t}= g_{\infty,t}$ in the $\BB^{k-1}(I)$-topology,
where the family $g_{\infty,t}$ is a $C^{k-1+Lip}$ family
of piecewise expanding $C^k$ unimodal maps.  Note that $g_{\infty,0}=f$. Since 
$J(g_{t,n},\partial_t g_{t,n})=0$
and
$$\lim_{j\to
\infty}\partial_t g_{n_j,t}|_{t=t_0}=\partial_t g_{\infty,t}|_{t=t_0},$$
in the  $\BB^0(I)$ topology, we conclude, by Corollary \ref{continuity}, that
$J( g_{\infty,t},\partial_s g_{\infty,s}|_{s=t})=0$.
By $D \Rightarrow B$ in Theorem \ref{equiv}, the map
$g_{\infty,t}$ is topologically conjugate 
with $f$.
By uniqueness, $g_{\infty,t}=\tilde f_t$ and $\lim_{n\to \infty}b_n=b$.
\end{proof}

 For $k\ge 2$,
 let $f\in \mathcal{U}^k(I)$ be a good map, and let $w\in \mathcal{B}^k(I)$  
 be such that $w(-1)=w(1)=0$ and $J(f,w)\neq 0$. Using a technique similar to the proof of Theorem \ref{create}, one can prove that there exist a neighborhood $\mathcal{V}_1$ of $f$  and a neighborhood $\mathcal{V}_2$ of $0$ in $\Ker J(f,\cdot)$ 
 (we consider $J(f,\cdot): \{ v\in \BB^k, v(-1)=v(1)=0\} \to \real$) such that each topological class in $\mathcal{V}_1$ is of the form
$\{ f+ v+\psi(v) w\colon \ v \in \mathcal{V}_2    \}$,
where $\psi$ is a   $C^{k-1+Lip}$ real-valued  function defined in $\mathcal{V}_2$. Moreover if $f_n\in \mathcal{V}_1$ for all integers $n\ge 0$,
with $\lim_{n\to \infty}f_n = f_\infty \in \mathcal{V}_1$, and if $\psi_{n}\colon \mathcal{V}_2\to \mathbb{R}$ defines the topological class of $f_n$ in $\mathcal{V}_1$, then $\lim_{n \to \infty} \psi_n= \psi_\infty$ in $C^{k-1}$ defines the topological
class of $f_\infty$.

We end by mentioning two immediate corollaries of Theorem \ref{create}.

\begin{cor} Let $f$ be a piecewise expanding $C^k$
 unimodal map and $v\in\BB^{k}(I)$, with $k\ge 2$, such that 
 $v(-1)=v(1)=0$ and
$J(f,v)=0$.
Then there exists a $C^{k-1+Lip}$ family of piecewise expanding $C^k$ unimodal maps 
$\tilde f_t$ such that $\tilde f_0=f$, $\partial_t \tilde f_t|_{t=0}=v$ and $\tilde  f_t$ 
is topologically conjugate with $f$, for every $t$.
\end{cor}
\begin{proof} Choose  $w\in \BB^k(I)$ with $w(-1)=w(1)=0$ such that $J(f,w)\neq 0$. Consider the family 
$f_t = f + tv$ and apply Theorem \ref{create}.
\end{proof}

\begin{cor} Let $f_t$ be a $C^{k}$ family of   piecewise
expanding $C^k$ unimodal maps, $k\geq 2$, such that $f_t$ 
is topologically conjugate with $f_0$, for every $t$. 
Then there  exists a sequence of $C^{k}$-families 
$t \in (-\delta,\delta) \mapsto g_{t,n}$
such that 
\begin{itemize} 
\item[-] The map $g_{t,n}$ is topologically conjugate with $g_{0,n}$, for every $t$.
\item[-]  The critical point of $g_{0,n}$ is periodic.
\item[-]  The families $g_{t,n}$ converge to the family $f_{t}$ in the $\BB^{k-1}$ topology.
\end{itemize} 
\end{cor}

\begin{proof} Apply Theorem \ref{create} to the family $f_t$, noting that
Theorem \ref{equiv} implies that
$J(f_0,\partial_t f_t|_{t=0})=0$. 
So there exists a 
unique family $\tilde f_t$ such that $\tilde f_t$ 
is topologically conjugate with $f_0$, for every small $t$.
We conclude that $f_t=\tilde f_t$. 
To finish, apply the last claim in Theorem \ref{create}.
\end{proof}



\begin{thebibliography}{dMvS93}

\bibitem{av} A. Avila,
\textit{Infinitesimal perturbations of rational maps,}
Nonlinearity {\bf 15} (2002) 695--704.

\bibitem{ALM} A. Avila, M. Lyubich, and W. de Melo,
\textit{Regular or stochastic dynamics in real analytic families of
unimodal maps,}
Invent. Math. {\bf 154} (2003) 451--550. 


\bibitem{Ba} V. Baladi,
\textit{On the susceptibility function of piecewise expanding interval
maps,} Comm. Math. Phys. {\bf 275} (2007) 839--859.

\bibitem{BS} V. Baladi and D. Smania,
\textit{Linear response formula for piecewise expanding unimodal maps,} Nonlinearity {\bf 21} (2008) 677--711.

\bibitem{Ke} G. Keller,
\textit{Stochastic stability in some chaotic dynamical systems,}
Monatshefte Math. {\bf 94} (1982) 313--333.


\bibitem{L} M. Lyubich,
\textit{ Feigenbaum-Coullet-Tresser universality and Milnor's hairiness conjecture.}
Ann. of Math. (2) {\bf 149} (1999) 319--420. 

\bibitem{MM} M. Mazzolena, \textit{Dinamiche espansive unidimensionali:
dipendenza della misura invariante da un parametro,} Master's Thesis, Roma 2, Tor Vergata (2007). 

\bibitem{MSS} R. Ma\~n\'e, P. Sad, and D. Sullivan,
\textit{ On the dynamics of rational maps,}
Ann. Sci. \'Ecole Norm. Sup. (4) {\bf 16 } (1983) 193--217. 

\bibitem{MS} C. McMullen and D. Sullivan,
\textit{Quasiconformal homeomorphisms and dynamics. III. The Teichm\"uller space of a holomorphic dynamical system,}
Adv. Math. {\bf 135} (1998)  351--395. 

\bibitem{T} M. Tsujii,
\textit{ Lyapunov exponents in families of one-dimensional dynamical systems,}
 Invent. Math.  {\bf 111} (1993) 113--137. 



\end{thebibliography}
\end{document}